\documentclass[12pt,twoside]{amsart}
\usepackage{amssymb}
\usepackage{verbatim}
\usepackage{amsmath}
\usepackage{bm}
\usepackage{a4wide}
\usepackage[latin1]{inputenc}
\usepackage[T1]{fontenc}
\usepackage{times}
\usepackage{amssymb,latexsym}
\usepackage{enumerate}
\usepackage[stable]{footmisc}
\usepackage{color}
\usepackage[colorlinks,linkcolor=black,citecolor=black,urlcolor=black]{hyperref}

\usepackage{ulem}
 \usepackage{cases}

\usepackage{fancyhdr}
\fancypagestyle{plain}{%
    \fancyhf{}%
    \cfoot{\thepage}%
}
\newcommand{\D}{\ensuremath{\mathbb{D}}}
\newcommand{\C}{\ensuremath{\mathbb{C}}}

\newcommand{\bea}{\begin{eqnarray*}}
\newcommand{\eea}{\end{eqnarray*}}

\makeatletter
\newcommand{\sumprime}{\if@display\sideset{}{'}\sum%
            \else\sum'\fi}
\makeatother

\numberwithin{equation}{section}

\newtheorem{theorem}{Theorem}[section]

\newtheorem{corollary}[theorem]{Corollary}
\newtheorem{lemma}[theorem]{Lemma}

\numberwithin{definition}{section} %\def\thedefinition{\unskip}
\newtheorem{remark}{Remark}

\title[Energy asymptotics of holomorphic functions with application to Calder\'{o}n-Zygmund theory in $\C$]{Energy asymptotics of holomorphic functions with application to Calder\'{o}n-Zygmund theory in $\C$}
\dedicatory{}

 \author{Hongrong Chen}
\address{School of Mathematics (Zhuhai), Sun Yat-Sen University, Zhuhai, Guangdong 519082, P. R.
  China}
 \email{chenhr73@mail2.sysu.edu.cn}

 \author{Guokuan Shao}
\address{School of Mathematics (Zhuhai), Sun Yat-Sen University, Zhuhai, Guangdong 519082, P. R.
  China}
 \email{shaogk@mail.sysu.edu.cn}

 \author{Jujie Wu}
\address{School of Mathematics (Zhuhai), Sun Yat-Sen University, Zhuhai, Guangdong 519082, P. R.
 	China}
\email{wujj86@mail.sysu.edu.cn}

\author{Wei Xia}
\address{School of Mathematics (Zhuhai), Sun Yat-Sen University, Zhuhai, Guangdong 519082, P. R.
  China}
 \email{xiaw27@mail2.sysu.edu.cn}

\thanks{This paper was supported by National Key R\&D Program of China, No. 2024YFA1015200. Guokuan Shao was partially supported by National Natural Science Foundation of China No.~12471082. Jujie Wu was partially supported by  National Natural Science Foundation of China No.~12671104  and  Natural Science Foundation of Guangdong Province, No. 2025A1515011428.}

\begin{document}

\begin{abstract}
The Calder\'on-Zygmund theory establishes the boundedness of singular integral operators on \(L^p\) spaces for \(1 < p < \infty\), yet it encounters a failure at the endpoint \(p = 1\). While radial counterexamples in \(\mathbb{R}^n\) are well-documented, Pan-Shao-Wang-Wu \cite{psww2026} has showed that every nonconstant holomorphic function provides a counterexample to the Poisson equation within the Calder\'on-Zygmund framework, with the singular locus being a complex subvariety of codimension one. In this paper, we focus on the complex one-dimensional case and establish stronger results. We prove asymptotic formulas with explicit constants for both the level-set integral and the sublevel-set energy. Then we give simplified proofs of the universal counterexamples to Calder\'on-Zygmund theory at \(p = 1\) in $\C$. Additionally, we construct a new family of counterexamples at the endpoint \(p = \infty\), showing that the failure of \(W^{2,\infty}\)-regularity is also a universal phenomenon in complex one dimension.
  \bigskip


  \smallskip

\end{abstract}

\maketitle

\tableofcontents

\section{Introduction}
The Calder\'on-Zygmund theory \cite{GT} stands as a cornerstone of modern harmonic analysis and partial differential equations. At its heart lies the systematic study of singular integral operators, which are fundamental to understanding the regularity of solutions to elliptic equations. A classic application is the Poisson equation
\[
\Delta u = g
\]
on a domain \(\Omega \subset \mathbb{R}^n\). In the classical H\"older setting, if \(g \in C^{k,\alpha}_\mathrm{loc}(\Omega)\), one expects \(u \in C^{k+2,\alpha}_\mathrm{loc}(\Omega)\). More generally, within the framework of Sobolev spaces, the theory establishes that for \(1 < p < \infty\), if \(g\) belongs to \(L_{\mathrm{loc}}^p(\Omega)\), then \(u\in W_{\mathrm{loc}}^{2,p}(\Omega)\).

This \(L^p\) regularity theory is indispensable in elliptic PDE, enabling existence and uniqueness results in settings where classical derivatives may fail. However, a fundamental limitation of the Calder\'on-Zygmund theory is its breakdown at the endpoint \(p = 1\). Singular integral operators, such as the Riesz transforms, are generally not bounded on \(L^1(\mathbb{R}^n)\). In the context of the Poisson equation, this means that \(g \in L_{\mathrm{loc}}^1(\Omega)\) does not guarantee \(u \in W_{\mathrm{loc}}^{2,1}(\Omega)\).

The standard textbook counterexample is the radial function
\[
u(x) = \log(-\log |x|^2)
\]
defined near the origin in \(\mathbb{R}^2\). Although the distributional Laplacian of \(u\) is represented by an \(L_{\mathrm{loc}}^1\)-function, its second-order derivatives are not locally integrable, exhibiting a blow-up of order \((|x|^2\log |x|^2)^{-1}\) near the origin. Although instructive, such radial examples have long fostered the impression that the \(L^1\) failure is a peculiarity of isolated point singularities.

A recent work by Pan-Shao-Wang-Wu \cite{psww2026} demonstrated that the failure of \(L^1\) regularity is a far more widespread phenomenon, deeply rooted in the geometry of complex varieties. They proved that every nonconstant holomorphic function \(f(z)\) on a domain \(\Omega \subset \mathbb{C}^n \cong \mathbb{R}^{2n}\) yields, via the iterated logarithmic transformation
\[
u(z) = \log(-\log |f(z) - f(z_0)|^2),
\]
a counterexample to the Calder\'on-Zygmund framework. In this example, the singular locus is not merely a point, but a complex subvariety of codimension one. This expands the catalogue of known counterexamples from simple radial functions to entire families of holomorphic functions, including those with cusps, nodes, and higher-order vanishing singularities.

Let $U$ be a unit ball or a polydisc in $\mathbb{C}^n$, and let $f$ be a nonconstant holomorphic function on an open neighborhood of $\overline{U}$. To establish the distributional validity of these counterexamples, one needs geometric estimates for the level sets \(\{|f| = \varepsilon\}\) and the sublevel sets \(\{|f| < \varepsilon\}\). In \cite[Theorem 1.1, Theorem 1.2]{psww2026}, sharp energy and volume estimates are established:
\[
\int_{\{z\in U: |f|=\varepsilon\}} |\partial f|\,dS = O(\varepsilon),\qquad \int_{\{z\in U: |f|<\varepsilon\}} |\partial f|^2\,dV = O(\varepsilon^2),
\]
as well as measure bounds for the corresponding level and sublevel sets:
\[
\mathcal{H}^{2n-1}(\{z\in U: |f|=\varepsilon\}) = O(\varepsilon^\gamma),\qquad \operatorname{Vol}(\{z\in U: |f|<\varepsilon\}) = O(\varepsilon^\tau),
\]
for some \(0 < \gamma \leq 1\) and \(0 < \tau \leq 2\). 

These questions are closely related to the Sobolev regularity of
logarithmic potentials associated with real-analytic sets. In particular, Shi and Zhang \cite{SZ2025} established that $\log\vert{}F\vert{} \in W^{1,1}_{\mathrm{loc}}$ near the zero locus of any real-analytic function $F$, provided its codimension is at least two. Conversely, Pan and Zhang \cite{PanZ2024} utilized Hironaka's resolution of singularities to construct local continuous $W^{1,1}_{\mathrm{loc}}$ functions whose weak gradients have removable singularities exactly on an arbitrary prescribed real-analytic set. Together, these developments underscore that the underlying geometry of analytic varieties fundamentally governs such critical Sobolev phenomena.

%Shi--Zhang
%\cite{SZ2025} proved that if $F$ is real analytic and the real codimension of
%its zero set is at least two, then $\log|F|\in W^{1,1}_{\mathrm{loc}}$ near the
%zero set. In a complementary direction, Pan--Zhang \cite{PanZ2024} used
%resolution of singularities to construct, near every point of an arbitrary
%real-analytic set, continuous functions in $W^{1,1}_{\mathrm{loc}}$ whose weak
%derivatives have removable singularities precisely on that set. These results
%further illustrate that the geometry of analytic sets plays a fundamental
%role in borderline Sobolev regularity. 

In this paper, we focus on the complex one-dimensional case \(\Omega \subset \mathbb{C}\). This setting exhibits a different structure: the zeros of holomorphic functions are isolated points, and the complex geometry simplifies considerably. This enables us to obtain sharp energy asymptotics and a more refined analysis of the failure of endpoint Calder\'on--Zygmund estimates. In contrast to the $O(\varepsilon)$ and $O(\varepsilon^2)$ upper bounds in higher dimensions, we establish precise asymptotic formulas with explicit constants for the energy integrals.

%This allows us to obtain sharp energy asymptotics and a more explicit analysis of the Calder\'on--Zygmund endpoint failures. Unlike the estimates \(O(\varepsilon)\) and \(O(\varepsilon^2)\) obtained in higher dimensions, we establish asymptotic formulas with explicit constants for the energy integrals, which is stronger than the case in higher dimensions. 

Given a domain $U$ in $\mathbb{C}$,  denote by $W ^{k,p}(U) $ (resp. $W_{\mathrm{loc}}^{k,p}(U)$) the Sobolev space of all functions on  $U$ whose weak derivatives of order $\leq k$ exist and belong to $L^p(U)$ (  resp. $L_{\mathrm{loc}}^p(U)$),   for some $k\in \mathbb{Z}^+,  p\geq 1$. Let $f$ be a holomorphic function in $U$  and $Z(f)= \{z \in U:  f(z) =0 \}$ be the zero set of $f$. We denote $ \partial_z f := \frac{\partial f}{\partial z}$,  $ \partial_{\bar{z}} f := \frac{\partial f}{\partial \bar{z}}$.   Here without loss of generality, we assume $Z(f) \neq \emptyset$ throughout the paper. The Laplacian of $u$ in $\mathbb{C}$ is defined as $\Delta u := 4 \frac{\partial^2 u}{\partial z \partial \bar{z}}.$ Let $\mathbb{D} := \{ z \in \mathbb{C} : |z| < 1 \}$. For $\varepsilon > 0$, set
$$
\Gamma_{\varepsilon} := \{z \in \mathbb{D} : |f(z)| = \varepsilon\}, \qquad
E_{\varepsilon} := \{z \in \mathbb{D} : |f(z)| < \varepsilon\}.
$$

\begin{theorem}\label{thm3.2}
Let $U \subset \mathbb{C}$ be an open neighborhood of $\overline{\mathbb{D}}$, and let $f \in \mathcal{O}(U)$ be a non-trivial holomorphic function. Let $p_1, \dots, p_m \in \mathbb{D}$ denote the zeros of $f$ in $\mathbb{D}$ with multiplicities $k_1, \dots, k_m$, and let $q_1, \dots, q_r \in \partial\mathbb{D}$ denote the boundary zeros of $f$ with multiplicities $\ell_1, \dots, \ell_r$. 
Then, as $\varepsilon \to 0^+$, we have the asymptotic expansions
\begin{equation}\label{eq:level-set-asymptotic}
\int_{\Gamma_{\varepsilon}} |f'(z)|\,ds = 2\pi \left( \sum_{j=1}^{m} k_{j} + \frac{1}{2} \sum_{\nu=1}^{r} \ell_{\nu} \right) \varepsilon + o(\varepsilon),
\end{equation}
and
\begin{equation}\label{eq:sublevel-energy-asymptotic}
\int_{E_{\varepsilon}} |f'(z)|^{2}\,dA = \pi \left( \sum_{j=1}^{m} k_{j} + \frac{1}{2} \sum_{\nu=1}^{r} \ell_{\nu} \right) \varepsilon^{2} + o(\varepsilon^{2}).
\end{equation}
\end{theorem}

These formulas reveal a clear geometric principle: each interior zero contributes its full multiplicity to the leading coefficient, whereas each boundary zero contributes exactly half of its multiplicity.

%\begin{theorem}\label{thm3.2}
%		Let \(U\) be an open neighborhood of \(\overline{\mathbb{D}}\), and let \(f \in \mathcal O(U)\) with \(f \not\equiv 0\). Let \(p_1, \dots, p_m \in \mathbb{D}\) be all zeros of \(f\) in $\D$, with multiplicities \(k_1, \dots, k_m\), and let \(q_1, \dots, q_r \in \partial\mathbb{D}\) be all zeros of \(f\) on the boundary, with multiplicities \(\ell_1, \dots, \ell_r\).
%Then, as $\varepsilon\to 0^{+}$, we have
%\begin{equation}\label{eq:level-set-asymptotic}
%\int_{\Gamma_{\varepsilon}} |f'(z)|\,ds = 2\pi \left( \sum_{j=1}^{m} k_{j} + \frac{1}{2} \sum_{\nu=1}^{r} \ell_{\nu} \right) \varepsilon + o(\varepsilon),
%\end{equation}
%and
%\begin{equation}\label{eq:sublevel-energy-asymptotic}
%\int_{E_{\varepsilon}} |f'(z)|^{2}\,dA = \pi \left( \sum_{j=1}^{m} k_{j} + \frac{1}{2} \sum_{\nu=1}^{r} \ell_{\nu} %\right) \varepsilon^{2} + o(\varepsilon^{2}).
%\end{equation}
%\end{theorem}

%These formulas reveal a geometric principle: each interior zero contributes its full multiplicity, while each boundary zero contributes exactly half of its multiplicity. %This precise information is inaccessible in higher dimensions, where the geometry of the zero set is far more complicated.

\medskip
While Theorem \ref{thm3.2} provides sharp weighted asymptotic formulas for the level and sublevel sets of $f$, establishing distributional identities subsequently requires unweighted size estimates for these geometric sets.

\begin{theorem}\label{thm3.3}
Under the assumptions and notation of Theorem \ref{thm3.2}, there exist positive constants $C_j, C_j'$ ($1 \le j \le m$) and $C_\nu, C_\nu'$ ($1 \le \nu \le r$) such that, for all sufficiently small $\varepsilon > 0$,
\[
\mathcal{H}^1(\Gamma_\varepsilon)
\le
\sum_{j=1}^m C_j \varepsilon^{1/k_j}
+
\sum_{\nu=1}^r C_\nu \varepsilon^{1/\ell_\nu},
\]
and
\[
\operatorname{Area}(E_\varepsilon)
\le
\sum_{j=1}^m C_j' \varepsilon^{2/k_j}
+
\sum_{\nu=1}^r C_\nu' \varepsilon^{2/\ell_\nu}.
\]
\end{theorem}
%The preceding theorem gives sharp weighted asymptotic formulas for the small level and sublevel sets of $f$. For the later applications to distributional identities, we shall also need unweighted estimates for the geometric size of these sets.
%\begin{theorem}\label{thm3.3}
%Under the assumptions and notation of  Theorem \ref{thm3.2},  there exist constants $C_j, C_j'> 0$, $1 \le j \le m$, and $C_\nu, C_\nu' > 0$, $1 \le \nu \le r$, such that, for all sufficiently small $\varepsilon > 0$,
%	\[
%		\mathcal H^1(\Gamma_\varepsilon)
%		\le
%		\sum_{j=1}^m C_j \varepsilon^{1/k_j}
%		+
%		\sum_{\nu=1}^r C_\nu \varepsilon^{1/\ell_\nu},
%		\]
%		and
%		\[
%		\operatorname{Area}(E_\varepsilon)
%		\le
%%		\sum_{j=1}^m C_j' \varepsilon^{2/k_j}
%		+
%		\sum_{\nu=1}^r C_\nu' \varepsilon^{2/\ell_\nu}.
%		\]
%	\end{theorem}

As an immediate consequence, we obtain the following uniform upper bounds, which will be utilized repeatedly.

\begin{corollary}\label{c01}
Assume that $m+r \ge 1$, and define
\[
\kappa := \max\{ k_1, \dots, k_m, \ell_1, \dots, \ell_r \}.
\]
Then there exist constants $C > 0$ and $\varepsilon_0 > 0$ such that, for all $0 < \varepsilon < \varepsilon_0$,
\[
\mathcal{H}^1(\Gamma_\varepsilon) \le C \varepsilon^{1/\kappa}
\qquad \text{and} \qquad
\operatorname{Area}(E_\varepsilon) \le C \varepsilon^{2/\kappa}.
\]
\end{corollary}

As an application of Theorems \ref{thm3.2} and \ref{thm3.3}, we present a self-contained proof of the failure of endpoint Calder\'on--Zygmund estimates in one dimension.

%Next we give a simple corollary, which will be used repeatedly.
%    \begin{corollary}\label{c01}
%		Assume that \(m+r \ge 1\), and let
%		$
	%	\kappa:= \max\{ k_1, \dots, k_m, \ell_1, \dots, \ell_r \}.
	%	$
	%	Then there exist constants \(C > 0\) and \(\varepsilon_0 > 0\) such that, for all \(0 < \varepsilon < %\varepsilon_0\),
	%	\[
	%	\mathcal H^1(\Gamma_\varepsilon) \le C \varepsilon^{1/\kappa},
	%	\qquad
	%	\operatorname{Area}(E_\varepsilon) \le C \varepsilon^{2/\kappa}.
	%	\]
%	\end{corollary}

%As an application of Theorem \ref{thm3.2} and Theorem \ref{thm3.3}, we give an alternative proof of the following result on universal counterexamples for Calder\'on-Zygmund estimates in the one-dimensional case. %without using the \L{}ojasiewicz inequality, the coarea formula and Hironaka's resolution, compared with \cite[Theorem 1.4]{psww2026}.\begin{theorem}\label{th:maintheorem1}

\begin{theorem}\label{th:maintheorem1}
Let $\Omega\subset\mathbb C$ be a domain, and let $f$ be a nonconstant holomorphic function satisfying
$
Z(f):=\{z\in\Omega:f(z)=0\}\neq\emptyset
$
and
$
|f|<\frac1{10}
$
on $\Omega$.
Define
\[
u(z):=\log\bigl(-\log|f(z)|^2\bigr),
\]
and
\begin{equation}
g_u(z):=
\begin{cases}
\Delta u, & z\notin Z(f),\\
0, & z\in Z(f).
\end{cases}
\end{equation}
\begin{enumerate}
\item $ u\in W_{\mathrm{loc}}^{1,2}(\Omega)$, but $u \notin  W_{\mathrm{loc}}^{1,p}(\Omega) $\ \text { for any } \ $p>2$ near $Z(f)$;
\item  $g_u \in  L_{\mathrm{loc}}^1(\Omega)$;
\item $\Delta u= g_u$ on $\Omega$  in the sense of distributions.  i. e., $u$ is a weak solution of $\Delta u= g_u$ on $\Omega$;
 \item Finally,  $u \notin  W_{\mathrm{loc}}^{2,1}(\Omega) $ near $Z(f)$.
\end{enumerate}
\end{theorem}

The next result concerns the reciprocal profile \(v=1/u\). In contrast to \(u\), this function belongs to \(W_{\mathrm{loc}}^{2,1}\), while failing to belong to \(W_{\mathrm{loc}}^{2,p}\) for any \(p>1\).

\begin{theorem}\label{th:maintheorem2}
Under the assumptions of Theorem \ref{th:maintheorem1}, define
\begin{equation}
v(z):=
\begin{cases}
\dfrac{1}{\log\bigl(-\log|f(z)|^2\bigr)}, & z\notin Z(f),\\[6pt]
0, & z\in Z(f).
\end{cases}
\end{equation}
 Then the following statements hold.
 \begin{enumerate}
\item $v$ is continuous on $\Omega$.  In particular, $v \in L^\infty (\Omega)$, and $v \in W_\mathrm{loc}^{1,2}(\Omega) $; but  $v \notin W_\mathrm{loc}^{1,p}(\Omega) $  \text { for any } \ $p>2$   near $Z(f)$;
\item  $g_v \in L^1_{\mathrm{loc}}(\Omega)$,  where
\begin{equation}
g_v(z):=
\begin{cases}
\Delta v, & z\notin Z(f),\\
0, & z\in Z(f).
\end{cases}
\end{equation}
  \item $\Delta v= g_v$ holds in the sense of distributions in $\Omega$.
 \item  $v\in W_\mathrm{loc}^{2,1}(\Omega) $,  but $v\notin W_\mathrm{loc}^{2,p}(\Omega) $ for any $p>1$ near $Z(f)$.
\end{enumerate}
 \end{theorem}

In higher dimensions, the proofs in \cite{psww2026} rely heavily on advanced machinery: the \L{}ojasiewicz gradient inequality to control $|\partial f|$ near the singular locus, the coarea formula to relate volume and boundary integrals, and Hironaka's resolution of singularities to reduce general analytic sets to normal crossings. In the one-dimensional setting, our approach bypasses these heavy tools entirely. Since the zeros of a non-trivial holomorphic function in $\mathbb{C}$ are isolated, the analysis becomes strictly local. Near each zero, writing $f(z) = (z-a)^m g(z)$ with $g(a) \neq 0$ allows a local biholomorphic coordinate change that reduces $f$ to a pure monomial $w^m$. Under this transformation, the level sets reduce to circular arcs and the sublevel sets to sectors or disks, allowing their geometric measures and weighted integrals to be computed explicitly. Together with Green's identity and standard removability techniques for Sobolev spaces, this provides an elementary yet sharp proof of the distributional identities.

\medskip
In addition to the endpoint $p = 1$, we construct a class of counterexamples to Calder\'on--Zygmund regularity at the opposite endpoint $p = \infty$, which is of independent interest.

\begin{theorem}\label{Linfty}
Let $U \subset \mathbb{C}$ be a bounded domain, and let $f \in \mathcal{O}(\overline{U})$ be a nonconstant holomorphic function such that $Z(f) \cap U = \{a_1, \dots, a_N\}$ is non-empty ($N \ge 1$). Define the polynomial
\[
P(z) := \prod_{j=1}^N (z-a_j),
\]
and set
\[
u(z) :=
\begin{cases}
P(z)^2 \log |f(z)|^2, & z \in U \setminus Z(f),\\[4pt]
0, & z \in Z(f) \cap U.
\end{cases}
\]
Then the following statements hold:
\begin{enumerate}
    \item $u \in C^1(U)$;
    \item $g \in L^\infty_{\mathrm{loc}}(U)$ and $\Delta u = g$ in the sense of distributions on $U$, where
    \begin{equation}
    g(z) :=
    \begin{cases}
    \Delta u(z), & z \in U \setminus Z(f),\\[4pt]
    0, & z \in Z(f) \cap U;
    \end{cases}
    \end{equation}
    \item $u \notin W^{2,\infty}_{\mathrm{loc}}(U)$.
\end{enumerate}
Consequently, $u$ serves as a one-dimensional counterexample to Calder\'on--Zygmund estimates at $p = \infty$.
\end{theorem}

The quadratic power in the prefactor $P(z)^2$ is critical: replacing $P^2$ with $P^q$ for any integer $q \ge 3$ would restore full $W^{2,\infty}_{\mathrm{loc}}(U)$ regularity, demonstrating that $q = 2$ is precisely the threshold for singularity formation.

\medskip
The remainder of this paper is organized as follows. In Section \ref{sec2}, we establish the sharp asymptotic formulas and geometric estimates for the level and sublevel sets of holomorphic functions (Theorems \ref{thm3.2} and \ref{thm3.3}). In Section \ref{sec3}, we prove the universal failure of Calder\'on--Zygmund estimates at $p = 1$ (Theorems \ref{th:maintheorem1} and \ref{th:maintheorem2}). Finally, Section \ref{sec4} is devoted to the construction and analysis of the counterexample at the $p = \infty$ endpoint (Theorem \ref{Linfty}).

\section{Asymptotics of energy estimates of holomorphic functions}\label{sec2}
While energy estimates for (sub)level sets of holomorphic functions were obtained in \cite{psww2026}, the fact that zeros are isolated in one complex dimension allows us to sharpen these bounds into asymptotic formulas. In particular, we derive the exact leading terms of the level-set integral and the sublevel-set energy, showing that interior and boundary zeros contribute their full and half multiplicities, respectively. To this end, we first establish an elementary geometric lemma bounding the length and area of small circular sections of planar domains with $C^1$ boundary.

%Energy estimates for the level and sublevel sets of holomorphic functions were established in \cite{psww2026}. In the one-dimensional setting of the present paper, the isolated nature of zeros enables us to sharpen these estimates into precise asymptotic formulas. Specifically, we explicitly compute the leading-order asymptotics for both the level-set gradient integral and the sublevel-set energy. These formulas reveal that each interior zero contributes its full multiplicity to the leading coefficient, whereas each boundary zero contributes exactly half its multiplicity.% In the one-dimensional setting considered in this paper, Because zeros are isolated in one complex dimension, these estimates can be sharpened to asymptotic formulas. More precisely, we compute the leading terms of the level-set integral and the sublevel-set energy. These formulas show that each interior zero contributes its full multiplicity, whereas each boundary zero contributes one half of its multiplicity. 
%We first prove a   geometric estimate for the length and area of small circular sections of a planar domain with $C^1$ boundary.

\begin{lemma}\label{lem3.1}
Let $\Omega \subset \mathbb{C}$ be a domain with $0 \in \partial\Omega$, and suppose that $\partial\Omega$ is of class $C^1$ near $0$. Define
\[
L_{\Omega}(\rho) := \mathcal{H}^1\bigl(\Omega \cap \{|w|=\rho\}\bigr)
\qquad \text{and} \qquad
A_{\Omega}(\rho) := \operatorname{Area}\bigl(\Omega \cap \{|w|<\rho\}\bigr).
\]
Then, as $\rho \to 0^+$,
\begin{equation}\label{eq:half-domain-asymptotics}
L_{\Omega}(\rho) = \pi\rho + o(\rho)
\qquad \text{and} \qquad
A_{\Omega}(\rho) = \frac{\pi}{2}\rho^2 + o(\rho^2).
\end{equation}
\end{lemma}

\begin{proof}
Up to a rotation, we may assume that the tangent line to $\partial\Omega$ at $0$ is the real axis, and that $\Omega$ locally lies in the upper half-plane. Since $\partial\Omega$ is of class $C^1$, there exist $r_0 > 0$ and a $C^1$ function $\varphi \colon (-r_0, r_0) \to \mathbb{R}$ with $\varphi(0) = \varphi'(0) = 0$ such that
\[
\Omega \cap B(0, r_0) = \{x + iy \in B(0, r_0) : y > \varphi(x)\}.
\]
Set $\delta(\rho) := \sup_{|x| \le \rho} |\varphi(x)| = o(\rho)$ as $\rho \to 0^+$. For $0 < \rho < r_0$, let $E_\rho := \{\theta \in [0, 2\pi) : \rho e^{i\theta} \in \Omega\}$, so that $L_\Omega(\rho) = \rho |E_\rho|$. 

Writing $w = \rho e^{i\theta} = x + iy$, the condition $y > \delta(\rho)$ implies $y > \varphi(x)$, whereas $y < -\delta(\rho)$ implies $y < \varphi(x)$. Consequently,
\[
\left\{\theta \in [0, 2\pi) : \sin\theta > \frac{\delta(\rho)}{\rho}\right\}
\subset E_\rho \subset
\left\{\theta \in [0, 2\pi) : \sin\theta \ge -\frac{\delta(\rho)}{\rho}\right\}.
\]
Since $\delta(\rho)/\rho \to 0$ as $\rho \to 0^+$, for $\rho$ sufficiently small we have $0 \le \delta(\rho)/\rho < 1$. Hence
$$
\left|\left\{ \theta \in [0, 2\pi) : \sin \theta > \frac{\delta(\rho)}{\rho} \right\}\right| = \pi - 2 \arcsin \left( \frac{\delta(\rho)}{\rho} \right),
$$
and
$$
\left|\left\{ \theta \in [0, 2\pi) : \sin \theta \ge -\frac{\delta(\rho)}{\rho} \right\}\right| = \pi + 2 \arcsin \left( \frac{\delta(\rho)}{\rho} \right).
$$
It follows that
\begin{equation}\label{eq:angular-measure-squeeze}
\pi-2\arcsin\!\left(\frac{\delta(\rho)}{\rho}\right)
\leq |E_\rho| \leq
\pi+2\arcsin\!\left(\frac{\delta(\rho)}{\rho}\right).
\end{equation}
Since $\delta(\rho)/\rho\to0$, it follows from
\eqref{eq:angular-measure-squeeze} that
\[
|E_\rho|=\pi+o(1),
\qquad \rho\to0^+.
\]
Consequently,
\[
L_\Omega(\rho)=\rho|E_\rho|=\pi\rho+o(\rho).
\]
		
For the area estimate, using polar coordinates gives
\[
A_{\Omega}(\rho)=\int_{0}^{\rho} r\,|E_{r}|\,dr.
\]
Since $|E_{r}|=\pi+o(1)$ as $r\to0^{+}$, we obtain
\[
A_{\Omega}(\rho)
=\int_{0}^{\rho} r(\pi+o(1))\,dr
=\frac{\pi}{2}\rho^{2}+o(\rho^{2}).
\]
	\end{proof}

We now apply this geometric lemma to the small level sets of a holomorphic function. Near an interior zero, the level set is biholomorphically equivalent to a full circle. Near a boundary zero, it is equivalent to the intersection of a small circle with a $C^1$ half-domain. This gives the refined asymptotic formula in Theorem \ref{thm3.2}.

\begin{proof}[Proof of Theorem \ref{thm3.2}]
Since $U$ is connected, $\overline{\mathbb{D}}$ is compact, and $f \not\equiv 0$, the identity theorem implies that $f$ has only finitely many zeros in $\overline{\mathbb{D}}$. We partition these zeros into the interior zeros $p_1, \dots, p_m \in \mathbb{D}$ and the boundary zeros $q_1, \dots, q_r \in \partial\mathbb{D}$.
%Since \(f\) is holomorphic on a neighborhood of \(\overline{\mathbb{D}}\) and $f \not\equiv 0$, it has only finitely many zeros in \(\overline{\mathbb{D}}\). These zeros split into two classes: the interior zeros  $p_{1}, \dots, p_{m} \in \mathbb{D}$, and the boundary zeros $q_{1}, \dots, q_{r} \in \partial\mathbb{D}$.

Choose pairwise disjoint open disks $B_1, \dots, B_m \Subset \mathbb{D}$ and $W_1, \dots, W_r \Subset U$, centered at $p_j$ and $q_\nu$ respectively, such that all $\{B_j\}_{j=1}^m$ and $\{W_\nu\}_{\nu=1}^r$ are mutually disjoint, and $f$ has no zeros on$$\bigcup_{j=1}^m \bigl(\overline{B_j} \setminus \{p_j\}\bigr) \cup \bigcup_{\nu=1}^r \bigl(\overline{W_\nu} \setminus \{q_\nu\}\bigr).$$The set$$K := \overline{\mathbb{D}} \setminus \Biggl( \bigcup_{j=1}^m B_j \cup \bigcup_{\nu=1}^r W_\nu \Biggr)$$is compact, and by construction, $f$ nowhere vanishes on $K$. Setting $\varepsilon_0 := \inf_{z \in K} \vert{}f(z)\vert{} > 0$, it follows that for every $\varepsilon \in (0, \varepsilon_0)$, the level set $\Gamma_\varepsilon$ and the sublevel set $E_\varepsilon$ are entirely contained in $\bigl(\bigcup_{j=1}^m B_j\bigr) \cup \bigl(\bigcup_{\nu=1}^r (W_\nu \cap \mathbb{D})\bigr)$. Consequently, the integrals in \eqref{eq:level-set-asymptotic} and \eqref{eq:sublevel-energy-asymptotic} decompose into the sum of the localized contributions around the interior and boundary zeros.

%For each interior zero \(p_j\), choose a small disk \(B_j \Subset U\) such that \(p_j \in B_j\), the disks \(B_j\) are pairwise disjoint, and \(f\) has no zeros on \(\overline{B_j} \setminus \{p_j\}\). For each boundary zero \(q_\nu\), choose a small disk \(W_\nu \Subset U\) such that \(q_\nu \in W_\nu\), the disks \(W_\nu\) are pairwise disjoint and disjoint from all \(B_j\), and \(f\) has no zeros on \(\overline{W_\nu} \setminus \{q_\nu\}\). Set
%$$
%K := \overline{\mathbb{D}} \setminus \left( \bigcup_{j=1}^m B_j \cup \bigcup_{\nu=1}^r W_\nu \right).
%$$
%Then $K$ is compact, $f$ has no zeros on $K$. Hence, there exists %$\varepsilon_0 > 0$ such that
%$
%|f(z)| \ge \varepsilon_0,
%$
%on \(K\). Therefore, for $0 < \varepsilon < \varepsilon_0$, the sets $\Gamma_\varepsilon$ and $E_\varepsilon$ are contained in
%$
%\biggl(\bigcup_{j=1}^m B_j \biggr)\cup \biggl(\bigcup_{\nu=1}^r W_\nu\biggr).
%$
%Thus, the desired integrals decompose into the sum of the contributions from the interior zero neighborhoods and the boundary zero neighborhoods.

\medskip
\noindent\textbf{Case 1: Interior zeros.} Fix $j \in \{1, \dots, m\}$. In the simply connected domain $B_j$, we may write $f(z) = (z - p_j)^{k_j} g_j(z)$ with non-vanishing $g_j \in \mathcal{O}(B_j)$. Selecting a holomorphic branch $h_j := g_j^{1/k_j} \in \mathcal{O}(B_j)$, the mapping$$\psi_j(z) := (z - p_j) h_j(z)$$satisfies $f(z) = \psi_j(z)^{k_j}$ and $\psi_j'(p_j) = h_j(p_j) \neq 0$. Shrinking $B_j$ if necessary, $\psi_j$ defines a biholomorphism onto a neighborhood of $0$. Setting $w = \psi_j(z)$ and $\rho_j := \varepsilon^{1/k_j}$, the pull-back metric relations yield$$ds_w = \vert{}\psi_j'(z)\vert{}\,ds_z, \qquad dA_w = \vert{}\psi_j'(z)\vert{}^2\,dA_z.$$Because $f'(z) = k_j w^{k_j - 1} \psi_j'(z)$, we have$$\vert{}f'(z)\vert{}\,ds_z = k_j \vert{}w\vert{}^{k_j-1}\,ds_w, \qquad \vert{}f'(z)\vert{}^2\,dA_z = k_j^2 \vert{}w\vert{}^{2k_j-2}\,dA_w.$$Under the coordinate $w$, the intersection $\Gamma_\varepsilon \cap B_j$ corresponds to the full circle $\{\vert{}w\vert{} = \rho_j\}$, while $E_\varepsilon \cap B_j$ corresponds to the disk $\{\vert{}w\vert{} < \rho_j\}$. Direct integration yields
\begin{equation}\label{eq:interior-zero-contributions}
\begin{aligned}\int_{\Gamma_\varepsilon \cap B_j} |f'(z)|,ds_z &= k_j \rho_j^{k_j-1} \cdot 2\pi \rho_j = 2\pi k_j \varepsilon, \\int_{E_\varepsilon \cap B_j} |f'(z)|^2,dA_z &= k_j^2 \int_0^{\rho_j} 2\pi t \cdot t^{2k_j-2},dt = \pi k_j \varepsilon^2.
\end{aligned}
\end{equation}
\medskip
\noindent\textbf{Case 2: Boundary zeros.}
Fix $\nu \in \{1, \dots, r\}$. Similarly, on $W_\nu$ we factorize $f(z) = (z - q_\nu)^{\ell_\nu} g_\nu(z)$ with non-vanishing $g_\nu \in \mathcal{O}(W_\nu)$, and define a local coordinate $\psi_\nu(z) := (z - q_\nu) g_\nu(z)^{1/\ell_\nu}$ such that $f(z) = \psi_\nu(z)^{\ell_\nu}$ and $\psi_\nu'(q_\nu) \neq 0$. Shrinking $W_\nu$ if necessary, $\psi_\nu$ is a biholomorphism from $W_\nu$ onto an open neighborhood of the origin. Consider the domain $\Omega_\nu := \psi_\nu(W_\nu \cap \mathbb{D})$. Since $q_\nu \in \partial\mathbb{D}$ and $\partial\mathbb{D}$ is real-analytic, the boundary $\partial\Omega_\nu$ is a $C^1$-smooth curve near $0 = \psi_\nu(q_\nu)$ whose tangent line at $0$ splits the tangent space into a half-plane. Setting $w = \psi_\nu(z)$ and $\rho_\nu := \varepsilon^{1/\ell_\nu}$, the set $\Gamma_\varepsilon \cap W_\nu$ transforms into $\Omega_\nu \cap \{\vert{}w\vert{} = \rho_\nu\}$, and $E_\varepsilon \cap W_\nu$ transforms into $\Omega_\nu \cap \{\vert{}w\vert{} < \rho_\nu\}$. For the level-set integral, using \eqref{eq:half-domain-asymptotics} we obtain$$\operatorname{length}\bigl( \Omega_\nu \cap \{\vert{}w\vert{} = \rho_\nu\} \bigr) = \pi \rho_\nu + o(\rho_\nu) \quad \text{as } \rho_\nu \to 0^+.$$
Therefore,
\begin{equation}\label{eq:boundary-level-contribution}
\int_{\Gamma_\varepsilon \cap W_\nu} |f'(z)|,ds_z = \ell_\nu \rho_\nu^{\ell_\nu - 1} \operatorname{length}\bigl(\Omega_\nu \cap {|w| = \rho_\nu}\bigr) = \pi \ell_\nu \varepsilon + o(\varepsilon).
\end{equation}
For the sublevel-set integral, let $w = t e^{i\theta}$. For $t > 0$ sufficiently small, define the angular slice
$$
\Theta_\nu(t) := \bigl\{\theta \in [0, 2\pi) : t e^{i\theta} \in \Omega_\nu\bigr\}.
$$
By Lemma~\ref{lem3.1}, the angular measure satisfies $\vert{}\Theta_\nu(t)\vert{} = \pi + o(1)$ as $t \to 0^+$. Passing to polar coordinates gives
$$  
\begin{aligned} 
\int_{E_\varepsilon \cap W_\nu} \vert{}f'(z)\vert{}^2\,dA_z &= \ell_\nu^2 \int_{\Omega_\nu \cap \{\vert{}w\vert{} < \rho_\nu\}} \vert{}w\vert{}^{2\ell_\nu - 2}\,dA_w \\ &= \ell_\nu^2 \int_0^{\rho_\nu} t^{2\ell_\nu - 1} \vert{}\Theta_\nu(t)\vert{}\,dt \\ &= \ell_\nu^2 \int_0^{\rho_\nu} t^{2\ell_\nu - 1} \bigl(\pi + o(1)\bigr)\,dt \\ &= \frac{\pi}{2} \ell_\nu \rho_\nu^{2\ell_\nu} + o(\rho_\nu^{2\ell_\nu}) \\ &= \frac{\pi}{2}\ell_\nu \varepsilon^2 + o(\varepsilon^2). 
\end{aligned}
$$
Hence,
\begin{equation}\label{eq:boundary-energy-contribution}
\int_{E_\varepsilon \cap W_\nu} |f'(z)|^2,dA_z = \frac{\pi}{2} \ell_\nu \varepsilon^2 + o(\varepsilon^2).
\end{equation}
\medskip
Finally, summing \eqref{eq:interior-zero-contributions} over $j = 1, \dots, m$ and \eqref{eq:boundary-level-contribution} over $\nu = 1, \dots, r$ yields \eqref{eq:level-set-asymptotic}. Likewise, summing the second identity in \eqref{eq:interior-zero-contributions} and \eqref{eq:boundary-energy-contribution} over all respective zeros yields \eqref{eq:sublevel-energy-asymptotic}.

\end{proof}

		$
%		Then
		$

We next prove Theorem \ref{thm3.3}.

\begin{proof}[Proof of Theorem \ref{thm3.3}]
We retain the notation and localization setup from the proof of Theorem~\ref{thm3.2}. Let $B_1, \dots, B_m \Subset \mathbb{D}$ and $W_1, \dots, W_r \Subset U$ be mutually disjoint open neighborhoods of the interior zeros $p_1, \dots, p_m$ and boundary zeros $q_1, \dots, q_r$, respectively, chosen such that $f$ has no other zeros in $\bigcup_{j=1}^m \overline{B_j} \cup \bigcup_{\nu=1}^r \overline{W_\nu}$. By compactness and non-vanishing of $f$ on the complement, there exists $\varepsilon_0 > 0$ such that for all $\varepsilon \in (0, \varepsilon_0)$,
\[
\Gamma_\varepsilon \cup E_\varepsilon \subset \bigcup_{j=1}^m B_j \cup \bigcup_{\nu=1}^r \bigl(W_\nu \cap \overline{\mathbb{D}}\bigr).
\]
It therefore suffices to establish localized upper bounds in each neighborhood.

\medskip
\noindent\textbf{Case 1: Interior zeros.} 
Fix $j \in \{1, \dots, m\}$. As established in the proof of Theorem~\ref{thm3.2}, after shrinking $B_j$ if necessary, there exists a biholomorphism $\psi_j \colon B_j \to \psi_j(B_j) \subset \mathbb{C}$ satisfying $\psi_j(p_j) = 0$ and $f(z) = \psi_j(z)^{k_j}$. For $\varepsilon > 0$, set $\rho_j := \varepsilon^{1/k_j}$. Under the local coordinate $w = \psi_j(z)$, we have
\[
\Gamma_\varepsilon \cap B_j = \psi_j^{-1}(\{|w| = \rho_j\}), \qquad E_\varepsilon \cap B_j = \psi_j^{-1}(\{|w| < \rho_j\}).
\]
Since $(\psi_j^{-1})'$ is continuous near the origin, there exist constants $\delta_j > 0$ and $M_j > 0$ such that $|(\psi_j^{-1})'(w)| \le M_j$ whenever $|w| < \delta_j$. Thus, for all $\varepsilon \in (0, \delta_j^{k_j})$, the change of variables $z = \psi_j^{-1}(w)$ gives
\[
\mathcal{H}^1(\Gamma_\varepsilon \cap B_j) = \int_{|w|=\rho_j} |(\psi_j^{-1})'(w)|\,ds_w \le M_j \int_{|w|=\rho_j} ds_w = 2\pi M_j \rho_j = C_j \varepsilon^{1/k_j},
\]
where $C_j := 2\pi M_j$. Similarly, for the area of the sublevel set,
\[
\operatorname{Area}(E_\varepsilon \cap B_j) = \int_{|w|<\rho_j} |(\psi_j^{-1})'(w)|^2\,dA_w \le M_j^2 \int_{|w|<\rho_j} dA_w = \pi M_j^2 \rho_j^2 = C_j' \varepsilon^{2/k_j},
\]
where $C_j' := \pi M_j^2$.

\medskip
\noindent\textbf{Case 2: Boundary zeros.} 
Fix $\nu \in \{1, \dots, r\}$. Shrinking $W_\nu$ if necessary, there exists a biholomorphism $\psi_\nu \colon W_\nu \to \psi_\nu(W_\nu)$ such that $\psi_\nu(q_\nu) = 0$ and $f(z) = \psi_\nu(z)^{\ell_\nu}$. Setting $\Omega_\nu := \psi_\nu(W_\nu \cap \mathbb{D})$ and $\rho_\nu := \varepsilon^{1/\ell_\nu}$, we identify
\[
\Gamma_\varepsilon \cap W_\nu = \psi_\nu^{-1}\bigl(\Omega_\nu \cap \{|w| = \rho_\nu\}\bigr), \qquad E_\varepsilon \cap W_\nu = \psi_\nu^{-1}\bigl(\Omega_\nu \cap \{|w| < \rho_\nu\}\bigr).
\]
By Lemma~\ref{lem3.1}, there exist $\delta_\nu > 0$ and geometric constants $\alpha_\nu, \beta_\nu > 0$ such that for all $\rho \in (0, \delta_\nu)$,
\[
\operatorname{length}\bigl(\Omega_\nu \cap \{|w| = \rho\}\bigr) \le \alpha_\nu \rho \qquad \text{and} \qquad \operatorname{Area}\bigl(\Omega_\nu \cap \{|w| < \rho\}\bigr) \le \beta_\nu \rho^2.
\]
By the continuity of $(\psi_\nu^{-1})'$ on a neighborhood of $0$, we may choose $M_\nu > 0$ such that $|(\psi_\nu^{-1})'(w)| \le M_\nu$ for all $|w| < \delta_\nu$ (after shrinking $\delta_\nu$ if necessary). Consequently, for all $\varepsilon \in (0, \delta_\nu^{\ell_\nu})$,
\[
\mathcal{H}^1(\Gamma_\varepsilon \cap W_\nu) = \int_{\Omega_\nu \cap \{|w|=\rho_\nu\}} |(\psi_\nu^{-1})'(w)|\,ds_w \le M_\nu \operatorname{length}\bigl(\Omega_\nu \cap \{|w| = \rho_\nu\}\bigr) \le C_\nu \varepsilon^{1/\ell_\nu},
\]
where $C_\nu := M_\nu \alpha_\nu$. In the same manner,
\[
\operatorname{Area}(E_\varepsilon \cap W_\nu) = \int_{\Omega_\nu \cap \{|w|<\rho_\nu\}} |(\psi_\nu^{-1})'(w)|^2\,dA_w \le M_\nu^2 \operatorname{Area}\bigl(\Omega_\nu \cap \{|w| < \rho_\nu\}\bigr) \le C_\nu' \varepsilon^{2/\ell_\nu},
\]
where $C_\nu' := M_\nu^2 \beta_\nu$.

\medskip
Finally,
taking $\varepsilon > 0$ sufficiently small so that $\varepsilon < \min\bigl\{\varepsilon_0, \min_j \delta_j^{k_j}, \min_\nu \delta_\nu^{\ell_\nu}\bigr\}$, summing the respective localized estimates over all $j = 1, \dots, m$ and $\nu = 1, \dots, r$ yields the desired bounds 
 \begin{equation}\label{eq:global-geometric-bounds}
 \begin{aligned}
 \mathcal H^1(\Gamma_\varepsilon)
 &\leq
 \sum_{j=1}^m C_j\varepsilon^{1/k_j}
 +\sum_{\nu=1}^r C_\nu\varepsilon^{1/\ell_\nu},\\
 \operatorname{Area}(E_\varepsilon)
 &\leq
 \sum_{j=1}^m C_j'\varepsilon^{2/k_j}
 +\sum_{\nu=1}^r C_\nu'\varepsilon^{2/\ell_\nu}.
 \end{aligned}
 \end{equation}
% This proves the asserted geometric bounds.
\end{proof}

We conclude this section with the proof of Corollary \ref{c01}.
\begin{proof}[Proof of Corollary \ref{c01}]

By \eqref{eq:global-geometric-bounds}, for all sufficiently small
$\varepsilon>0$,
\[
\mathcal H^1(\Gamma_\varepsilon)
\leq
\sum_{j=1}^m C_j\varepsilon^{1/k_j}
+\sum_{\nu=1}^r C_\nu\varepsilon^{1/\ell_\nu},
\]
and the analogous estimate holds for
$\operatorname{Area}(E_\varepsilon)$.
Since $k_{j} \leq \kappa$ and $\ell_{
\nu} \leq \kappa$, for $0 < \varepsilon < 1$ we have
\[
\varepsilon^{1/k_{j}} \leq \varepsilon^{1/\kappa}, \qquad \varepsilon^{1/\ell_{
\nu}} \leq \varepsilon^{1/\kappa},
\]
and similarly,
\[
\varepsilon^{2/k_{j}} \leq \varepsilon^{2/\kappa}, \qquad \varepsilon^{2/\ell_{
\nu}} \leq \varepsilon^{2/\kappa}.
\]
Thus, after increasing the positive constant $C$ if necessary, we have
\[
\mathcal{H}^{1}(\Gamma_{\varepsilon}) \leq C \varepsilon^{1/\kappa}, \qquad \operatorname{Area}(E_{\varepsilon}) \leq C \varepsilon^{2/\kappa},
\]
for all sufficiently small $\varepsilon > 0$.
	\end{proof}

\section{Universal counterexamples at the $p = 1$ endpoint }\label{sec3}

This section is devoted to the proofs of Theorems \ref{th:maintheorem1} and \ref{th:maintheorem2}. In the proof of Theorem \ref{th:maintheorem1}, the level-set asymptotics established in  Section 2  serve to control the boundary terms arising from Green's formula, thereby allowing us to extend the classical identity for $u$ across the zero set $Z(f)$ in the distributional sense. The proof of Theorem \ref{th:maintheorem2} subsequently builds upon the first-order regularity of $u$ combined with a second-order removability principle. For clarity of presentation, we begin by assembling several preliminary differential identities and auxiliary results that will be used throughout the proofs.

%This section is devoted to the proofs of Theorem \ref{th:maintheorem1} and Theorem \ref{th:maintheorem2}. The level-set estimates established in Section 2 are used in the proof of Theorem \ref{th:maintheorem1} to control the boundary terms arising from Green's formula and to extend the classical identity for $u$ across $Z(f)$ in the sense of distributions. The proof of Theorem \ref{th:maintheorem2} then builds on the first-order regularity of $u$, together with a second-order removability principle. We begin by collecting the differential identities and the auxiliary results needed in the two proofs.

\subsection{Differential identities and removability results}\label{Differential}

Throughout this section, let $\Omega \subset \mathbb{C}$ and $f \in \mathcal{O}(\Omega)$ satisfy the assumptions of Theorem \ref{th:maintheorem1}. Let $u$ and $v$ be the functions introduced in Theorem \ref{th:maintheorem1}  and Theorem \ref{th:maintheorem2}, respectively. We first record the differential identities for $u$. On $\Omega\setminus Z(f)$, the function $u$ is smooth, and direct
differentiation gives
\begin{equation}\label{eq:u-first-derivatives}
\frac{\partial u}{\partial z}
=\frac{f'}{f\log|f|^2},
\qquad
\frac{\partial u}{\partial\bar z}
=\frac{\overline{f'}}{\overline f\log|f|^2}.
\end{equation}
Differentiating once more, we obtain
\begin{equation}\label{eq:u-pure-second-derivative}
\frac{\partial^2u}{\partial z^2}
=
\frac{f''}{f\log|f|^2}
-\left(\frac{f'}f\right)^2\frac1{\log|f|^2}
-\left(\frac{f'}f\right)^2\frac1{(\log|f|^2)^2},
\qquad
\frac{\partial^2u}{\partial\bar z^2}
=\overline{\frac{\partial^2u}{\partial z^2}}.
\end{equation}
Moreover,
\begin{equation}\label{eq:u-laplacian}
\Delta u
=4\frac{\partial^2u}{\partial z\partial\bar z}
=-4\left|\frac{f'}f\right|^2\frac1{(\log|f|^2)^2}
=-4\left|\frac{\partial u}{\partial z}\right|^2.
\end{equation}

We next turn to $v$. On $\Omega\setminus Z(f)$, direct computation gives
\begin{equation}\label{eq:v-first-derivatives}
\frac{\partial v}{\partial z}
=-\frac{f'}f
\frac1{\log|f|^2\bigl(\log(-\log|f|^2)\bigr)^2},
\qquad
\frac{\partial v}{\partial\bar z}
=-\frac{\overline{f'}}{\overline f}
\frac1{\log|f|^2\bigl(\log(-\log|f|^2)\bigr)^2}.
\end{equation}
Differentiating once more, we obtain
\begin{equation}\label{eq:v-pure-second-derivative}
\begin{aligned}
\frac{\partial^2v}{\partial z^2}
&=-\frac{f''}{f}
  \frac1{\log|f|^2\bigl(\log(-\log|f|^2)\bigr)^2}\\
&\quad+\left(\frac{f'}f\right)^2
  \frac1{\log|f|^2\bigl(\log(-\log|f|^2)\bigr)^2}\\
&\quad+\left(\frac{f'}f\right)^2
  \frac1{(\log|f|^2)^2\bigl(\log(-\log|f|^2)\bigr)^2}\\
&\quad+2\left(\frac{f'}f\right)^2
  \frac1{(\log|f|^2)^2\bigl(\log(-\log|f|^2)\bigr)^3},\\
\frac{\partial^2v}{\partial\bar z^2}
&=\overline{\frac{\partial^2v}{\partial z^2}}.
\end{aligned}
\end{equation}

Finally, the Laplacian is given by
\begin{equation}\label{eq:v-laplacian}
\begin{aligned}
\Delta v
&=4\left|\frac{f'}f\right|^2
\frac1{(\log|f|^2)^2\bigl(\log(-\log|f|^2)\bigr)^2}\\
&+8\left|\frac{f'}f\right|^2
\frac1{(\log|f|^2)^2\bigl(\log(-\log|f|^2)\bigr)^3}.
\end{aligned}
\end{equation}

The identities above hold classically on $\Omega \setminus Z(f)$. To extend them across the isolated zero set in the sense of distributions, we now recall the required removability results.
\begin{lemma} [{\cite[Lemma A.1]{PZ2023}}]
\label{lm:ruojieinonedimension}
Let $f:\D  \rightarrow \C$ and $f\in L^1(\D)$. If $u: \D \rightarrow \C$ with $u\in L^2(\D)$ satisfies
$$
 \overline{\partial} u = f     \    \  \text{in} \ \  \   \D \setminus \{ 0\}
$$
in the sense of distributions,  then
$$
 \overline{\partial} u = f     \    \  \text{in} \ \  \   \D
$$
in the sense of distributions.
\end{lemma}
\begin{remark}
Note that the $L^2$ condition  of $u$  is sharp, which cannot be weakened further,  as shown in the following example.
 Let
\[
u_0(z)=
\begin{cases}
\dfrac{1}{z}, & z \neq 0; \\[4pt]
0, & z=0.
\end{cases}
\]
 Then $ \overline{\partial} u_0=0$ in $\D \setminus \{0 \}$. However,   $u_0$ is not a weak solution to $ \overline{\partial} u_0 = 0$ in $\D$.
\end{remark}
%Another proof technique,  developed by Pan and Zhang \cite{PanZ2024},  fundamentally relies on Hironaka's theorem on resolution of singularities.  We state this as the following lemma.
 \begin{lemma} \label{le:pacificidea1}
 Let $\Omega\subset\mathbb C$ be a domain. Given a nonconstant  holomorphic function $f$ with $| f| < \frac{1}{10}$ on $\Omega$. Set $u:=\log(-\log |f|^2) $ and assume
$
Z(f) \neq \emptyset.
$
Then
 $
u\in W^{1,2}_{\mathrm{loc}}(\Omega)$.
 \end{lemma}

%Given a nonconstant  holomorphic function $f$ on $\Omega\subset \C$  with $| f| < \frac{1}{10}$ on $\Omega$. Suppose that  the zero set of $f$ is non-empty,
%then $\log ( - \log |f|^2) \in W^{1,2}_{loc} (\Omega)$.
%\begin{enumerate*}
%\item .
 % \item If $f^{-1}(0) \neq 0$,  then  $\log ( - \log |f|^2) \notin \mathrm{W_{loc}^{1,p}}(\Omega)$ \ \text { for any } \ $p>2$;
%\end{enumerate*}
%Let $u  =\log ( - \log |f|^2)$, and
%\begin{eqnarray}
%g : = \frac{ f'}{ f}\cdot  \frac{1}{ \log (|f| ^2 )  }  \ \ \ \ \text{on} \ \   U .
%\end{eqnarray}
 %Then $g\in L^1_{loc}(U)$,  One has
% \begin{eqnarray}
%\frac{ \partial u }{\partial z}= g  \ \ \ \ \text{on} \ \   U .
%\end{eqnarray}
% in the sense of distribution.

\begin{proof}
Since $f$ is holomorphic and not identically zero on $\Omega$, its zero set $Z(f)$ consists of isolated points, and $u = \log(-\log |f|^2)$ is smooth on $\Omega \setminus Z(f)$. Note that the condition $|f| < \frac{1}{10}$ ensures $-\log|f|^2 > \log(100) > 1$, so $u$ is well-defined and positive throughout $\Omega \setminus Z(f)$.

We first verify that $u \in L^2_{\mathrm{loc}}(\Omega)$. Since $Z(f)$ is discrete, it suffices to analyze the local behavior near each zero $a \in Z(f)$. Without loss of generality, we may assume $a = 0$ with multiplicity $m \ge 1$. In a sufficiently small open disk $V \Subset \Omega$ centered at $0$, we can factorize $f(z) = z^m g(z)$, where $g \in \mathcal{O}(V)$ is nowhere vanishing. Then there exist constants $0 < c_1 < c_2 < \infty$ such that
\[
c_1 |z|^m \le |f(z)| \le c_2 |z|^m \quad \text{for all } z \in V \setminus \{0\}.
\]
Consequently, as $z \to 0$,
\[
u(z) = \log\bigl(-2m\log|z| - \log|g(z)|^2\bigr) = \log|\log|z|| + O(1).
\]
Since $\log|\log|z|| \in L^p_{\mathrm{loc}}(V)$ for every $p \in [1, \infty)$, it follows immediately that $u \in L^p_{\mathrm{loc}}(\Omega)$ for all $p \in [1, \infty)$; in particular, $u \in L^2_{\mathrm{loc}}(\Omega)$.

Next, we establish the existence and $L^2_{\mathrm{loc}}$-integrability of the distributional derivatives of $u$. On $\Omega \setminus Z(f)$, the classical Wirtinger derivatives of $u$ are given by
\begin{equation}\label{eq:u-first-derivatives}
\frac{\partial u}{\partial \bar{z}} = -\frac{\overline{f'(z)}}{\overline{f(z)} \log|f(z)|^2} \qquad \text{and} \qquad \frac{\partial u}{\partial z} = -\frac{f'(z)}{f(z) \log|f(z)|^2}.
\end{equation}
Define the measurable function $h \colon \Omega \to \mathbb{C}$ by
\[
h(z) := 
\begin{cases}
\dfrac{\partial u}{\partial \bar{z}}(z), & z \in \Omega \setminus Z(f), \\[6pt]
0, & z \in Z(f).
\end{cases}
\]
To show that $h$ represents the distributional derivative $\partial_{\bar{z}} u$ on $\Omega$, fix a zero $a = 0 \in Z(f)$ of multiplicity $m \ge 1$, and choose a simply connected neighborhood $V \Subset \Omega$ of $0$ such that $V \cap Z(f) = \{0\}$. On $V$, we introduce the local holomorphic coordinate 
\[
w = \psi(z) := z g(z)^{1/m},
\]
so that $f(z) = w^m$ and $\psi(0) = 0$. Since $\psi'(0) = g(0)^{1/m} \neq 0$, shrinking $V$ if necessary guarantees that $\psi \colon V \to \psi(V)$ is a biholomorphism, and both $|\psi'|$ and $|(\psi^{-1})'|$ are uniformly bounded above and below on $V$. In particular, there exists $r_0 \in (0, 1)$ such that $\psi(V) \subset \mathbb{D}_{r_0} := \{w \in \mathbb{C} : |w| < r_0\}$. 

Under this change of variables, $|f'/f|\,dA_z \asymp m|w|^{-1}\,dA_w$ and $\log|f|^2 = 2m\log|w|$. Thus,
\begin{equation}\label{eq:lp_for_h}
\int_V |h|^2\,dA_z = \int_{V \setminus \{0\}} \left| \frac{\partial u}{\partial \bar{z}} \right|^2 dA_z \lesssim \int_{\mathbb{D}_{r_0}} \frac{1}{|w|^2 |\log |w||^2}\,dA_w = 2\pi \int_0^{r_0} \frac{dr}{r |\log r|^2} < \infty,
\end{equation}
which implies $h \in L^2_{\mathrm{loc}}(V)$. By H\"older's inequality, we also have $h \in L^1_{\mathrm{loc}}(V)$. 

Since $u \in L^1_{\mathrm{loc}}(\Omega)$, $h \in L^1_{\mathrm{loc}}(\Omega)$, and $u$ is smooth outside the isolated set $Z(f)$, an application of Lemma~\ref{lm:ruojieinonedimension} ensures that
\[
\frac{\partial u}{\partial \bar{z}} = h \quad \text{on } \Omega
\]
holds in the sense of distributions, and hence $\partial_{\bar{z}} u \in L^2_{\mathrm{loc}}(\Omega)$. 

Finally, since $u$ is real-valued, the weak derivative $\partial_z u$ also exists and satisfies $\partial_z u = \overline{\partial_{\bar{z}} u}$ in the distributional sense, which directly gives $\partial_z u \in L^2_{\mathrm{loc}}(\Omega)$. Consequently, $|\nabla u| \in L^2_{\mathrm{loc}}(\Omega)$, which proves that $u \in W^{1,2}_{\mathrm{loc}}(\Omega)$.
\end{proof}

\begin{lemma}\label{lm:ruojieinonedimension2}
Let $\mathbb D$ be the unit disk and
$
P=\sum_{|\alpha|\le 2}a_\alpha(z)D^\alpha
$
be a linear differential operator of order two on $\mathbb D$, with coefficients $a_\alpha\in C^\infty(\mathbb D)$. Suppose that $u\in C(\mathbb D)$, $f\in L^1_{\mathrm{loc}}(\mathbb D)$, and
\[
Pu=f\ \ \text{on} \quad \mathbb  \D \setminus \{ 0\}
\]
in the sense of distributions. Then
\[
Pu=f \ \ \text{on} \quad \mathbb D
\]
in the sense of distributions.
\end{lemma}

\begin{proof}
Setting $\widetilde{u} := u - u(0)$ and $\widetilde{f} := f - u(0)P(1)$, we observe that $\widetilde{u}(0) = 0$ and $P\widetilde{u} = \widetilde{f}$ holds on $\mathbb{D} \setminus \{0\}$ in the sense of distributions. Since establishing the assertion for the pair $(\widetilde{u}, \widetilde{f})$ immediately implies the result for $(u, f)$, we may assume without loss of generality that $u(0) = 0$ by relabeling. Let $P^*$ denote the formal adjoint of $P$. To prove that $Pu = f$ on $\mathbb{D}$ in the distributional sense, it suffices to show that for every test function $\varphi \in C_c^\infty(\mathbb{D})$ (or $C_0^\infty(\mathbb{D})$),
\begin{equation}\label{eq:secondwant}
\int_{\mathbb{D}} u , P^*\varphi , dA = \int_{\mathbb{D}} f \varphi , dA.
\end{equation}
%Set
%$\widetilde u:=u-u(0),\widetilde f:=f-u(0)P(1)$. Then $\widetilde u(0)=0$ and
%$
%P\widetilde u=\widetilde f$ on $\D \setminus \{ 0\}$
%in the sense of distributions. Since the conclusion for $(\widetilde u,\widetilde f)$ immediately implies the conclusion for $(u,f)$, after relabeling $\widetilde u$ and $\widetilde f$ as $u$ and $f$, respectively, we may assume without loss of generality that
%$ u(0)=0$. Let $P^*$ denote the formal adjoint of $P$ with respect to the distributional pairing. It remains to show that, for every test function $\varphi\in C_0^\infty(\mathbb D)$,
%\begin{equation} \label{eq:secondwant}
%\int_{\mathbb D}u\,P^*\varphi\,dA = \int_{\mathbb D}f\varphi\,dA.
%\end{equation}
For $0<r<1$, choose $\chi_r\in C_0^\infty(\mathbb D)$ such that
\[
0\le \chi_r\le1,\qquad
\chi_r\equiv1\ \text{on }\mathbb D_{r/2},\qquad
\operatorname{supp}\chi_r\subset\mathbb D_r,
\]
and
\[
|\nabla\chi_r|\le\frac{C}{r},\qquad
|D^2\chi_r|\le\frac{C}{r^2},
\]
where $C>0$ is independent of $r$. Then, for every test function $\varphi\in C_0^\infty(\mathbb D)$,
the function $(1-\chi _r) \varphi$ is a test function on $\D \setminus \{ 0\}$. Hence,
\[
\int_{\mathbb D} u\,P^*\bigl((1-\chi_r)\varphi\bigr)\,dA
= \int_{\mathbb D} f(1-\chi_r)\varphi\,dA.
\]
Equivalently,
\begin{equation}\label{eq:seconddist}
\begin{aligned}
\int_{\mathbb D}uP^*\varphi\,dA
-\int_{\mathbb D}uP^*(\chi_r\varphi)\,dA
&=
\int_{\mathbb D}f\varphi\,dA
-\int_{\mathbb D}f\chi_r\varphi\,dA.
\end{aligned}
\end{equation}
It suffices to prove that
\[
\lim_{r\to 0}\int_{\mathbb D} u\,P^*(\chi_r\varphi)\,dA=0,
\qquad
\lim_{r\to 0}\int_{\mathbb D} f\chi_r\varphi\,dA=0.
\]
Once these limits are established, letting $r\rightarrow 0$ in (\ref{eq:seconddist}), we obtain the desired equality ($\ref{eq:secondwant}$). Since $f\in L^1_{\mathrm{loc}}(\mathbb D)$,
\begin{align*}
\left |\int_{ \D}  f \chi _r \varphi\,dA  \right| \leq C  \int _{|z|<r} |f|\,dA  \rightarrow 0,  \ \   r\rightarrow 0.
\end{align*}
On the other hand,   since $u$ is continuous on $\D$ and $u(0)=0$,
\begin{align*}
\left|\int_{\mathbb D}u\,P^*(\chi_r\varphi)\,dA\right|\le \frac{C}{r^2}\int_{|z|<r}|u(z)|\,dA   \rightarrow 0,  \ \   r\rightarrow 0.
\end{align*}
This completes the proof.
\end{proof}

%\begin{remark}
%The assumption of continuity can not be removed in Lemma \ref{lm:ruojieinonedimension2}.  For example, the following function $u(x)$ on $[-1,1]$, where
%\begin{numcases}{u(x)=}
%x, & $x<0;$  \\  \nonumber
%  0, & $x\geq 0$.
%\end{numcases}
%It is obvious that $u''(x) =0$ for every $x\neq 0.$  But $u''(0) = \delta_0$,  in the sense of distributions, where $\delta_0$ is the dirac measure.
%It should be noted that such $u$ is $L^\infty$, but not continuous. On the other hand, continuity of $u$ is not necessary, as shown the case in  the proof of Theorem \ref{th:maintheorem1},  $\Delta u = g$ in the senses of distributions.
%\end{remark}

\begin{remark}
The continuity assumption in Lemma \ref{lm:ruojieinonedimension2} cannot simply be omitted. Indeed, consider the particular case $P=\Delta$ and let
\[
u(z):=\log |z|,\qquad z\in\mathbb D\setminus\{0\},
\]
with an arbitrary value assigned at $z=0$. Then $u\in L^1_{\mathrm{loc}}(\mathbb D)$ and
\[
\Delta u=0 \quad \text{on} \quad \mathbb  \D \setminus \{ 0\}
\]
in the classical sense. However,
\[
\Delta u=2\pi\delta_0 \quad\text{on }\quad \mathbb D
\]
in the sense of distributions. Thus the conclusion of Lemma \ref{lm:ruojieinonedimension2} may fail if the continuity assumption is replaced merely by local integrability. On the other hand, continuity is not necessary in some particular cases. As shown in the proof of Theorem \ref{th:maintheorem1}, the identity
 $
\Delta u=g_u
 $
may still hold in the sense of distributions even when  $u $ is not continuous.
\end{remark}

\subsection{Proof of Theorem \ref{th:maintheorem1}}

\begin{proof}[Proof of Theorem \ref{th:maintheorem1}]
(1) The assertion $u \in W^{1,2}_{\mathrm{loc}}(\Omega)$ has been established in Lemma~\ref{le:pacificidea1}. It remains to show that $u \notin W_{\mathrm{loc}}^{1,p}(\Omega)$ near $Z(f)$ for any $p > 2$. 

Fix $a \in Z(f)$. Translating if necessary, we may assume without loss of generality that $a = 0$ with vanishing multiplicity $m \ge 1$. As in the proof of Lemma~\ref{le:pacificidea1}, there exist an open neighborhood $U \Subset \Omega$ of the origin, a radius $r_0 \in (0, 1)$, and a biholomorphism $\psi \colon U \to \psi(U) \supset \mathbb{D}_{r_0} := \{w \in \mathbb{C} : |w| < r_0\}$ such that $\psi(0) = 0$ and $f(z) = \psi(z)^m = w^m$. In the local coordinate $w = \psi(z)$, we have
\[
\frac{\partial u}{\partial z}(z) = -\frac{f'(z)}{f(z) \log|f(z)|^2} = -\frac{m \psi'(z)}{w \log|w|^{2m}}.
\]
Since $\psi'(0) \neq 0$, shrinking $U$ if necessary ensures that $|\psi'| \asymp 1$ on $U$. Applying the change of variables $w = \psi(z)$, we obtain
\begin{equation}\label{eq:u-W1p-divergence}
\int_U \left| \frac{\partial u}{\partial z} \right|^p dA_z \gtrsim \int_{\mathbb{D}_{r_0}} \frac{dA_w}{|w|^p |\log|w||^p} = 2\pi \int_0^{r_0} \frac{dr}{r^{p-1} |\log r|^p} = \infty \qquad (p > 2).
\end{equation}
Hence, $u \notin W^{1,p}(U)$ for every $p > 2$. Since $a \in Z(f)$ was arbitrary, $u$ fails to belong to $W^{1,p}_{\mathrm{loc}}$ on any open set intersecting $Z(f)$.

\medskip
\noindent (2) On $\Omega \setminus Z(f)$, direct differentiation via \eqref{eq:u-laplacian} gives
\[
\Delta u = -4 \left| \frac{f'(z)}{f(z)} \right|^2 \frac{1}{\bigl(\log|f(z)|^2\bigr)^2} = -4 \left| \frac{\partial u}{\partial z} \right|^2.
\]
By part (1), $\partial_z u \in L^2_{\mathrm{loc}}(\Omega)$, which implies that $|\partial_z u|^2 \in L^1_{\mathrm{loc}}(\Omega)$. Since $Z(f)$ is a discrete set of zero Lebesgue measure, the pointwise definition of $g_u$ on $Z(f)$ is negligible in $L^1_{\mathrm{loc}}(\Omega)$, and we conclude that $g_u \in L_{\mathrm{loc}}^1(\Omega)$.

\medskip
\noindent (3) We next show that $\Delta u = g_u$ on $\Omega$ in the sense of distributions. Since the claim is local, it suffices to examine an isolated zero $a = 0 \in Z(f)$. Let $U \Subset \Omega$ be an open disk centered at $0$ such that $\overline{U} \cap Z(f) = \{0\}$. We must verify that
\begin{equation}\label{eq:theo1-3}
\int_U u \Delta \varphi \, dA = \int_U g_u \varphi \, dA \qquad \text{for all } \varphi \in C_c^\infty(U).
\end{equation}
If $0 \notin \operatorname{supp} \varphi$, the identity holds trivially since $u \in C^\infty(U \setminus \{0\})$ and $\Delta u = g_u$ pointwise on $U \setminus \{0\}$. 

Now suppose $0 \in \operatorname{supp} \varphi$. For $\varepsilon > 0$ sufficiently small, define the sublevel and level sets
\[
K_{\varepsilon, U} := \{z \in U : |f(z)| \le \varepsilon\}, \qquad \Gamma_{\varepsilon, U} := \{z \in U : |f(z)| = \varepsilon\}.
\]
For all sufficiently small $\varepsilon$, $K_{\varepsilon, U} \Subset U$ and $\Gamma_{\varepsilon, U}$ is a smooth compact 1-manifold. Applying Green's second identity on $U \setminus K_{\varepsilon, U}$ yields
\begin{equation}\label{eq:green}
\int_{U \setminus K_{\varepsilon, U}} \bigl( u \Delta \varphi - \varphi \Delta u \bigr) \, dA = - \int_{\Gamma_{\varepsilon, U}} u \frac{\partial \varphi}{\partial n} \, ds + \int_{\Gamma_{\varepsilon, U}} \varphi \frac{\partial u}{\partial n} \, ds,
\end{equation}
where $n$ denotes the outward unit normal to $\partial K_{\varepsilon, U} = \Gamma_{\varepsilon, U}$, and the boundary terms on $\partial U$ vanish identically because $\operatorname{supp} \varphi \Subset U$.

As $\varepsilon \to 0^+$, the dominated convergence theorem implies that the left-hand side of \eqref{eq:green} converges to $\int_U u \Delta \varphi \, dA - \int_U g_u \varphi \, dA$. It therefore suffices to prove that both boundary integrals on the right-hand side vanish asymptotically.

First, by the length estimate in Theorem~\ref{thm3.3},
\[
\left| \int_{\Gamma_{\varepsilon, U}} u \frac{\partial \varphi}{\partial n} \, ds \right| \le \|\nabla \varphi\|_{L^\infty(U)} \log(-\log \varepsilon^2) \, \mathcal{H}^1(\Gamma_{\varepsilon, U}) \le C \|\nabla \varphi\|_{L^\infty(U)} \varepsilon^{1/m} \log(-\log \varepsilon^2) \longrightarrow 0
\]
as $\varepsilon \to 0^+$, where $m \ge 1$ is the multiplicity of $f$ at $0$.

Second, noting that $|\partial_n u| \le |\nabla u| = \frac{2|f'(z)|}{|f(z)| |\log|f(z)|^2|}$, we have
\[
\left| \int_{\Gamma_{\varepsilon, U}} \varphi \frac{\partial u}{\partial n} \, ds \right| \le \frac{2\|\varphi\|_{L^\infty(U)}}{\varepsilon |\log \varepsilon^2|} \int_{\Gamma_{\varepsilon, U}} |f'(z)| \, ds.
\]
Applying the asymptotic expansion for the level-set integral from Theorem~\ref{thm3.2},
\[
\int_{\Gamma_{\varepsilon, U}} |f'(z)| \, ds = 2\pi m \varepsilon + o(\varepsilon),
\]
which immediately yields
\[
\left| \int_{\Gamma_{\varepsilon, U}} \varphi \frac{\partial u}{\partial n} \, ds \right| \le \frac{C}{|\log \varepsilon|} \longrightarrow 0 \qquad \text{as } \varepsilon \to 0^+.
\]
Taking $\varepsilon \to 0^+$ in \eqref{eq:green} confirms \eqref{eq:theo1-3}, and hence $\Delta u = g_u$ on $\Omega$ distributionally.

\medskip
\noindent (4) Finally, we show that $u \notin W_{\mathrm{loc}}^{2,1}(\Omega)$ near $Z(f)$. In the biholomorphic coordinate $w = \psi(z)$ near $0 \in Z(f)$ with $f(z) = w^m$, let $\widetilde{u}(w) := u(\psi^{-1}(w))$. A direct calculation shows that on $\mathbb{D}_{r_0} \setminus \{0\}$,
\begin{equation}\label{eq:4.12}
\frac{\partial^2 \widetilde{u}}{\partial w^2} = -\frac{1}{w^2 \log|w|^2} - \frac{1}{w^2 (\log|w|^2)^2} = -\frac{1 + \log|w|^2}{w^2 (\log|w|^2)^2}.
\end{equation}
Choosing $r_0 \in (0, e^{-1})$ small enough so that $\log|w|^2 \le -2$ for all $0 < |w| < r_0$, we have $|1 + \log|w|^2| \ge \frac{1}{2}|\log|w|^2|$. Thus,
\[
\left| \frac{\partial^2 \widetilde{u}}{\partial w^2} \right| \ge \frac{1}{4 |w|^2 |\log|w||}.
\]
By the chain rule for second-order derivatives under conformal mappings, $$\partial_{z}^2 u(z) = \frac{\partial^2 \widetilde{u}}{\partial w^2}(\psi(z)) (\psi'(z))^2 + \frac{\partial \widetilde{u}}{\partial w}(\psi(z)) \psi''(z).$$ Because $\psi'(0) \neq 0$ and the first-derivative term is locally integrable in $L^1$, the lower bound of the transformation satisfies
\[
\int_U \left| \frac{\partial^2 u}{\partial z^2} \right| dA_z \gtrsim \int_{\mathbb{D}_{r_0}} \left| \frac{\partial^2 \widetilde{u}}{\partial w^2} \right| dA_w \ge \frac{1}{4} \int_{\mathbb{D}_{r_0}} \frac{dA_w}{|w|^2 |\log|w||} = \frac{\pi}{2} \int_0^{r_0} \frac{dr}{r |\log r|} = \infty.
\]
Consequently, $\partial_z^2 u \notin L_{\mathrm{loc}}^1(U)$, which proves that $u \notin W_{\mathrm{loc}}^{2,1}(\Omega)$ in any neighborhood of $Z(f)$.
\end{proof}

\subsection{Proof of Theorem \ref{th:maintheorem2}}

\begin{proof}[Proof of Theorem \ref{th:maintheorem2}]

Since $v=u^{-1}$ on $\Omega\setminus Z(f)$ and \(u(z)\to+\infty\) as \(z\to Z(f)\), the extension of $v$ by zero on $Z(f)$ is continuous on $\Omega$.
Moreover, since $|f|<1/10$ on $\Omega$, we have
$
u=\log(-\log|f|^2)>\log(\log 100)>0
$
on $\Omega\setminus Z(f)$, and hence $v\in L^\infty(\Omega)$.
We now prove the remaining assertions.

$(1)$  In order to prove $v\in W^{1,2}_{\mathrm{loc}}(\Omega)$,  first we need to show that the weak derivatives $\frac{\partial v}{\partial z}$ and $\frac{\partial v}{\partial \overline{z}}$ both exist in $\Omega$.  Since $\frac{\partial v}{\partial z}$ and $\frac{\partial v}{\partial \overline{z}}$ are conjugates of each other,  it suffices to prove that the weak derivative $\frac{\partial v}{\partial \overline{z}}$ exists. By \eqref{eq:v-first-derivatives}, on $\Omega\setminus Z(f)$ we have
\[
\frac{\partial v}{\partial \overline{z}}
=
-{v^2}\frac{\partial u}{\partial \overline{z}}
=
-\frac{\overline{f'}}{\overline{f}}\,
\frac{1}{\log |f|^2}\,
\frac{1}{\bigl(\log(-\log |f|^2)\bigr)^2}.
\]
Define
\[
h(z):=
\begin{cases}
\dfrac{\partial v}{\partial\bar z}, & z\notin Z(f),\\[4pt]
0, & z\in Z(f).
\end{cases}
\]

By Theorem \ref{th:maintheorem1}, we have
$
\frac{\partial u}{\partial \overline{z}}
\in L_{\mathrm{loc}}^2(\Omega).
$
Since \(v\) is bounded on \(\Omega\), it follows that
\[
h\in L_{\mathrm{loc}}^2(\Omega)
\subset L_{\mathrm{loc}}^1(\Omega).
\]
We claim that
$\frac{\partial v}{\partial \overline{z}}=h$ holds on $\Omega$ in the sense of distributions. Indeed, fix \(a\in Z(f)\), after translation, we may assume that $a = 0$. Choose a disk \(U\Subset\Omega\), centered at \(0\), such that \(0\) is the only zero of
\(f\) in \(U\). On \(U\setminus\{0\}\), the identity
$
\frac{\partial v}{\partial\bar z}=h
$
holds classically. Since $v$ is continuous on $U$, it clearly belongs to $L^2_{\mathrm{loc}}(U)$. Then, by Lemma \ref{lm:ruojieinonedimension}, the identity $\frac{\partial v} {\partial \overline{z}}=h$ holds on $U$ in the sense of distributions.  Since $a\in Z(f)$ was arbitrary and the identity already holds classically on $\Omega\setminus Z(f)$, it follows that
$\frac{\partial v}{\partial\bar z}=h$ on $\Omega$ in the sense of distributions.  Similarly, the weak derivative $ \frac{ \partial v} {\partial z} $ exists and $ \frac{ \partial v} {\partial z}\in L^2_{\mathrm{loc}}(\Omega)$.  Thus, $v\in W^{1,2}_{\mathrm{loc}}(\Omega)$.

It remains to prove the failure of $W^{1,p}_{\mathrm{loc}}$-regularity for any $p>2$. Select $a\in Z(f)$, we again use the same local coordinate construction as in the proof of $(1)$ in Theorem \ref{th:maintheorem1}. Set
$
\widetilde v(w):=v\bigl(\psi^{-1}(w)\bigr).
$
In the \(w\)-coordinate, we have
\[
\begin{aligned}
\int_U
\left|
\frac{\partial v}{\partial z}
\right|^p dA_z
&\gtrsim
\int_{|w|<r_0}
\left|
\frac{\partial\widetilde v}{\partial w}
\right|^p dA_w\\
&\approx
\int_{|w|<r_0}
\frac{1}
{|w|^p|\log|w||^p
\bigl(\log|\log|w||\bigr)^{2p}}dA_w\\
&=
2\pi\int_0^{r_0}
\frac{1}
{r^{p-1}|\log r|^p
\bigl(\log|\log r|\bigr)^{2p}}dr.
\end{aligned}
\]
For $p>2$, the last integral diverges. Thus $v\notin W^{1,p}_{\mathrm{loc}}(U)$ for any $p>2$. Since $a\in Z(f)$ was arbitrary, $v$ fails to belong to $W^{1,p}_{\mathrm{loc}}$ in every neighborhood of each point of $Z(f)$. This proves  $(1)$.

$(2)$ Recall that
$
u=\log\bigl(-\log|f|^2\bigr).
$
On $\Omega\setminus Z(f)$, Combining \eqref{eq:v-laplacian} with $v=u^{-1}$, we obtain
\begin{align} \label{computation}
g_v=\Delta v
&=
4\left|\frac{f'}{f}\right|^2
\frac{1}{\bigl(\log|f|^2\bigr)^2u^2}
+
8\left|\frac{f'}{f}\right|^2
\frac{1}{\bigl(\log|f|^2\bigr)^2u^3}\\ \nonumber
&=
4\left|
\frac{\partial u}{\partial z}
\right|^2
\left(
{v^2}+2{v^3}
\right).
\end{align}
By Theorem \ref{th:maintheorem1}, we have
$
\frac{\partial u}{\partial z}
\in L_{\mathrm{loc}}^2(\Omega).
$
Moreover, we know that $v$ is bounded on $\Omega$, and hence $v^{2}$ and $v^{3}$ are bounded on $\Omega$. It follows from \eqref{computation} that
$
g_v\in L_{\mathrm{loc}}^1(\Omega).
$
 This proves $(2)$.

$(3)$ Since the assertion is local, fix $a\in Z(f)$, and choose a disk $U\Subset\Omega$, centered at $a$, such that $a$ is the only zero of $f$ in $U$. On $U\setminus\{a\}$, the identity
$
\Delta v=g_v
$
holds in the classical sense. Moreover,
$v\in C(U)$ and $g_v\in L_{\mathrm{loc}}^1(U)$.
After translation and rescaling, Lemma \ref{lm:ruojieinonedimension2} applied with $P=\Delta$ yields
$\Delta v=g_v$ on $U$
in the sense of distributions. Since $a\in Z(f)$ was arbitrary and the identity holds classically on $\Omega\setminus Z(f)$, we conclude that
$
\Delta v=g_v$ in $\Omega$ in the sense of distributions. This proves $(3)$.

$(4)$  By part \textup{(3)} of  Theorem \ref{th:maintheorem2}, we only need to verify the weak derivatives $\frac{\partial ^2v}{ \partial z \partial z} $,  $\frac{\partial ^2 v}{ \partial \overline{z} \partial \overline{z}} $ exist  on $\Omega$  and are in $L^1_{\mathrm{loc}}$ near $Z(f)$.  Since they are mutually conjugate,  it is sufficient to check the weak derivatives  $\frac{\partial ^2 v}{ \partial z \partial z} $.

On \(\Omega\setminus Z(f)\), by \eqref{eq:v-pure-second-derivative}, write

\begin{equation}
\begin{split}
\frac{\partial^2 v}{\partial z^2}
&= - \frac{f''}{f}\frac{1}{\log(|f|^2)\bigl(\log(-\log|f|^2)\bigr)^2} \\
&\quad + \left( \frac{f'}{f}\right)^2
   \frac{1}{\log(|f|^2)\bigl(\log(-\log|f|^2)\bigr)^2} \\
&\quad + \left( \frac{f'}{f}\right)^2
   \frac{1}{\bigl(\log(|f|^2)\bigr)^2 \bigl(\log(-\log|f|^2)\bigr)^2} \\
&\quad + 2 \left( \frac{f'}{f}\right)^2
   \frac{1}{\bigl(\log(|f|^2)\bigr)^2 \bigl(\log(-\log|f|^2)\bigr)^3} \\
&= A_1 + A_2 + A_3 + A_4.
\end{split}
\label{eq:2nd_derivative_of _v}
\end{equation}
Define 
\[
h(z):=
\begin{cases}
\dfrac{\partial^2v}{\partial z^2}(z),
& z\in\Omega\setminus Z(f),\\[8pt]
0,
& z\in Z(f).
\end{cases}
\]
Since \(v\) is continuous on \(\Omega\), we have
$
v\in L_{\mathrm{loc}}^1(\Omega).
$
To apply Lemma \ref{lm:ruojieinonedimension2} with \(P=\frac{\partial^2 }{\partial z^2} \), by \eqref{eq:2nd_derivative_of _v} it remains to show that
$
A_1+A_2+A_3+A_4\in L_{\mathrm{loc}}^1(\Omega).
$
Note that the terms $A_3$ and $A_4$ are pointwise bounded by
\begin{eqnarray*}
|A_3|, |A_4| \leq C \left| \frac{f' }{f \log(|f|^2)}\right|^2.
\end{eqnarray*}
By $(1)$ of Theorem \ref{th:maintheorem1}, we obtain $A_3, A_4\in L^1_{\mathrm{loc}}(\Omega)$. To treat the pure second-order derivative, fix $a\in Z(f)$ and use the same
local coordinate $w=\psi(z)$ as in the proof of Lemma~3.2, so that
$f(z)=w^m$. Set
$\widetilde v(w):=v\bigl(\psi^{-1}(w)\bigr).$
A direct computation gives

\begin{align*}
\frac{\partial^2 \widetilde v}{\partial w^2}
&= 
\frac{m}
{w^2\log |w|^{2m}
 \bigl[\log(-\log |w|^{2m})\bigr]^2}
+
\frac{m^2}
{w^2(\log |w|^{2m})^2
 \bigl[\log(-\log |w|^{2m})\bigr]^2} \\
&+
\frac{2m^2}
{w^2(\log |w|^{2m})^2
 \bigl[\log(-\log |w|^{2m})\bigr]^3}.
\end{align*}

Hence, after decreasing $r_0>0$ if necessary,
\[
\left|
\frac{\partial^2 \widetilde v}{\partial w^2}
\right|
\le
\frac{C}
{|w|^2|\log |w||
 \bigl[\log(-\log |w|)\bigr]^2},
\qquad 0<|w|<r_0.
\]
Therefore,
\[
\begin{aligned}
\int_{|w|<r_0}
\left|
\frac{\partial^2 \widetilde v}{\partial w^2}
\right|\,dA_w
&\lesssim
\int_0^{r_0}
\frac{dr}
{r|\log r|[\log(-\log r)]^2} \\
&<\infty.
\end{aligned}
\]
Thus the classical derivative
$\partial^2\widetilde v/\partial w^2$ belongs to
$L^1_{\mathrm{loc}}$. Since $\widetilde v$ is continuous across $w=0$,
Lemma~\ref{lm:ruojieinonedimension2}, applied with
$P=\partial^2/\partial w^2$, shows that this classical derivative
coincides with the distributional derivative on a neighborhood of $0$.

Since the coordinate change is biholomorphic, local $W^{2,1}$-regularity
is invariant under this change of coordinates. Together with part~(3),
which gives
$\frac{\partial^2 v}{\partial z\,\partial\bar z} \in L^1_{\mathrm{loc}}(\Omega),$
and the conjugate pure derivative, we conclude that
$v\in W^{2,1}_{\mathrm{loc}}(U)$ near $Z(f)$.

Finally, we show that
$
v\notin W_{\mathrm{loc}}^{2,p}(\Omega)$ for any $ p>1$ near \(Z(f)\). 
To the contrary, suppose that there exist $p > 1$, $a \in Z(f)$, and a neighborhood $U \Subset \Omega$ of $a$ such that
$
v \in W^{2,p}_{\mathrm{loc}}(U).
$
Since the pure complex second derivative is a finite linear combination of the real second-order derivatives, it follows that
$
\frac{\partial^2 v}{\partial z^2}\in L^p(U).
$
Since local Sobolev regularity is invariant under biholomorphic diffeomorphisms, by the same local coordinate construction as above, the function in the new coordinates,
$
\widetilde v(w):=v\bigl(\psi^{-1}(w)\bigr),
$
must satisfy
$\frac{\partial^2 \widetilde{v}}{\partial w^2}
\in L_{}^p\bigl(\psi(U)\bigr).
$

Summing $A_1$ and $A_2$ reveals a crucial algebraic cancellation:
\[
A_1+A_2
=
\frac{m^2-m(m-1)}
{w^2\log |w|^{2m}
\left(\log\left(-\log |w|^{2m}\right)\right)^2}
=
\frac{m}
{w^2\log |w|^{2m}
\left(\log\left(-\log |w|^{2m}\right)\right)^2}.
\]
Since \(m\geq 1\), the numerator is nonzero. Thus the leading \(w^{-2}\) singularities in \(A_1\) and \(A_2\) do not cancel completely. Moreover,
\[
\bigg|\frac{A_3}{A_1+A_2}\bigg|
=
\frac{1}{|\log |w|^2|}
\longrightarrow 0
\qquad \text{as } w\to 0,
\]
and
\[
\bigg|\frac{A_4}{A_1+A_2}\bigg|
=
\frac{2}
{|\log |w|^2\,
\log(-\log |w|^{2m})|}
\longrightarrow 0
\qquad \text{as } w\to 0.
\]
Hence \(A_3\) and \(A_4\) are strictly lower-order terms compared with
\(A_1+A_2\). After decreasing \(r_0\) if necessary, we may therefore assume that
$
|A_3+A_4|
\leq
\frac{1}{2}|A_1+A_2|$
for $ 0<|w|<r_0.
$
It follows that
$
\left|
\frac{\partial^2 \widetilde v}{\partial w^2}
\right|
=
|A_1+A_2+A_3+A_4|
\geq
\frac{1}{2}|A_1+A_2|.
$
Therefore, there exists a constant \(C>0\) such that
\[
\left|
\frac{\partial^2 \widetilde v}{\partial w^2}
\right|^p
\geq
\frac{C}
{|w|^{2p}\,|\log |w||^p\,
\bigl|\log(-\log |w|^2)\bigr|^{2p}}
\]
for all \(0<|w|<r_0\).
Then, we obtain
\[
\int_{\{|w|<r_0\}}
\left|
\frac{\partial^2 \widetilde{v}}{\partial w^2}
\right|^p
\,dA_w
\gtrsim
\int_0^{r_0}
\frac{1}
{r^{2p-1}|\log r|^p
\bigl|\log(-\log r^2)\bigr|^{2p}}dr.
\]
For any $p>1$, the last integral diverges. Consequently,
$
\int_{\psi(U)}
\left|
\frac{\partial^2 \widetilde{v}}{\partial w^2}
\right|^p
\,dA_w
=
\infty,
$
which contradicts  our initial assumption that all  second-order weak derivatives are locally in $L^p$.
Therefore,
$
v\notin W^{2,p}_{\mathrm{loc}}(U)$ for every $ p>1.$ Since \(a\in Z(f)\) was arbitrary, \(v\) fails to belong to \(W^{2,p}_{\mathrm{loc}}\) in any neighborhood of each point of \(Z(f)\). This completes the proof of \((4)\).

\end{proof}

\section{A complementary counterexample at the $p=\infty$ endpoint}\label{sec4}

In this section, we prove Theorem \ref{Linfty}, which gives a complementary counterexample construction at the endpoint $p = \infty$. For every holomorphic function satisfying its assumptions, the construction produces a function $u \in C^1(U)$ and a function $g \in L^\infty_{\mathrm{loc}}(U)$ such that
$
\Delta u = g
$
in the sense of distributions, whereas
$
u \notin W^{2,\infty}_{\mathrm{loc}}(U).
$
We also show that the exponent $2$ in the factor $P^2$ is critical.

\begin{proof}[Proof of Theorem \ref{Linfty}]

%\medskip
%\noindent\textbf{Proof of (1).}
We prove the three assertions in order.

$(1)$ Fix $a_j\in Z(f)\cap U$, and let $m_j\ge1$ be the multiplicity of $a_j$ as a zero of $f$.
In a small neighborhood of $a_j$, we may write
$$
f(z)=(z-a_j)^{m_j}h_j(z),
\qquad h_j(a_j)\neq0.
$$
Since $P$ has a simple zero at $a_j$, we may also write
$$
P(z)=(z-a_j)Q_j(z),
\qquad Q_j(a_j)\neq0,
$$
where $Q_j$ is holomorphic near $a_j$. Hence $P(z)^2=(z-a_j)^2Q_j(z)^2$.

Near $a_j$, we have
$$
P(z)^2\log|f(z)|^2
=
(z-a_j)^2Q_j(z)^2
\left(
m_j\log|z-a_j|^2+\log|h_j(z)|^2
\right).
$$
The function $(z-a_j)^2\log|z-a_j|^2$ is $C^1$ near $a_j$. Therefore
$P(z)^2\log|f(z)|^2$ extends as a $C^1$ function across $a_j$.
Since this argument applies to every zero $a_j\in Z(f)\cap U$, and $u$ is smooth on
$U\setminus Z(f)$, we obtain $u\in C^1(U)$.

%\medskip
%\noindent\textbf{Proof of (2).}
$(2)$ On $U\setminus Z(f)$, since $P^2$ is holomorphic, we have
$$
\frac{\partial u}{\partial\bar z}
=
P(z)^2\frac{\overline{f'(z)}}{\overline{f(z)}}.
$$
Differentiating with respect to $z$, and using that
$\overline{f'(z)}/\overline{f(z)}$ is anti-holomorphic on $U\setminus Z(f)$, we get
$$
\frac{\partial^2 u}{\partial z\partial\bar z}
=
2P(z)P'(z)\frac{\overline{f'(z)}}{\overline{f(z)}}.
$$
Therefore, on $U\setminus Z(f)$,
$$
\Delta u
=
8P(z)P'(z)\frac{\overline{f'(z)}}{\overline{f(z)}}.
$$
Thus, $g$ is exactly the classical Laplacian of $u$ on $U\setminus Z(f)$.

We claim that $g\in L^\infty_{\mathrm{loc}}(U)$. It suffices to check boundedness near each zero
$a_j$. Near $a_j$, write $P(z)=(z-a_j)Q_j(z)$ with $Q_j(a_j)\neq0$. Then
$P(z)P'(z)=O(|z-a_j|)$. Moreover,
$$
\frac{\overline{f'(z)}}{\overline{f(z)}}
= \frac{m_j}{\overline{z - a_j}} + \frac{\overline{h'_j(z)}}{\overline{h_j(z)}}.
$$
Thus,
$$
P(z)P'(z)\frac{\overline{f'(z)}}{\overline{f(z)}}=O(1)
$$
near $a_j$. Away from $Z(f)\cap U$, the function $g$ is smooth. Hence,
$g\in L^\infty_{\mathrm{loc}}(U)$.

It remains to prove that $\Delta u=g$ holds on $U$ in the sense of distributions. Since
the statement is local, it is enough to check that no extra Dirac mass appears at each zero.
Fix $a_j$, and choose $r>0$ such that $D(a_j,r)\Subset U$ and $D(a_j,r)$ contains no zero of
$f$ other than $a_j$. Let $\eta\in C_c^\infty(D(a_j,r))$. For $0<\varepsilon<r$, set
$\Omega_\varepsilon:=D(a_j,r)\setminus\overline{D(a_j,\varepsilon)}$. Since $u$ is smooth on
$\Omega_\varepsilon$ and satisfies $\Delta u=g$ there, Green's identity gives
$$
\int_{\Omega_\varepsilon}u\Delta\eta\,dA
-
\int_{\Omega_\varepsilon}g\eta\,dA
=
\int_{|z-a_j|=\varepsilon}
\left(
u\frac{\partial\eta}{\partial\nu}
-
\eta\frac{\partial u}{\partial\nu}
\right)ds.
$$
From the expression in the proof of (1), as $z\to a_j$ we have
$u(z)=O(|z-a_j|^2|\log|z-a_j||)$ and
$|\nabla u(z)|=O(|z-a_j||\log|z-a_j||)$. Hence, on $|z-a_j|=\varepsilon$,
$$
|u(z)|\le C\varepsilon^2|\log\varepsilon|,
\qquad
|\nabla u(z)|\le C\varepsilon|\log\varepsilon|.
$$
Since the length of the circle $|z-a_j|=\varepsilon$ is $2\pi\varepsilon$, the boundary integral
is bounded by
$$
C\varepsilon^3|\log\varepsilon|+C\varepsilon^2|\log\varepsilon|,
$$
which tends to $0$ as $\varepsilon\to0^+$. Letting $\varepsilon\to0^+$ gives
$$
\int_{D(a_j,r)}u\Delta\eta\,dA
=
\int_{D(a_j,r)}g\eta\,dA.
$$
Thus $\Delta u=g$ distributionally across $a_j$. Since $a_j$ was arbitrary, $\Delta u=g$ holds
on $U$ in the sense of distributions.

%\medskip
%\noindent\textbf{Proof of (3).}
$(3)$
Fix $a_j\in Z(f)\cap U$, and write $\xi=z-a_j$. As above, near $a_j$ we have
$f(z)=\xi^{m_j}h_j(z)$ with $h_j(a_j)\neq0$, and $P(z)=\xi Q_j(z)$ with
$Q_j(a_j)\neq0$. Set $R_j(z):=Q_j(z)^2$. Then $P(z)^2=\xi^2R_j(z)$ and
$R_j(a_j)\neq0$.

On a punctured neighborhood of $a_j$, a direct computation gives
\begin{align}
\frac{\partial^2 u}{\partial z^2}
&=
(P^2)''(z)\log|f(z)|^2
+
2(P^2)'(z)\frac{f'(z)}{f(z)}
+
P(z)^2
\left(
\frac{f''(z)}{f(z)}
-
\left(\frac{f'(z)}{f(z)}\right)^2
\right).
\label{eq:second-derivative-u-theorem13}
\end{align}
The last two terms in \eqref{eq:second-derivative-u-theorem13} are bounded near $a_j$.
Indeed, $(P^2)'(z)=O(|\xi|)$ and $f'(z)/f(z)=m_j/\xi+O(1)$, while $P(z)^2=O(|\xi|^2)$ and
$$
\frac{f''(z)}{f(z)}
-
\left(\frac{f'(z)}{f(z)}\right)^2
=
-\frac{m_j}{\xi^2}+O(1).
$$
On the other hand, $(P^2)''(a_j)=2R_j(a_j)\neq0$, and
$$
\log|f(z)|^2
=
m_j\log|\xi|^2+O(1).
$$
Therefore,
\begin{align}
\frac{\partial^2 u}{\partial z^2}
=
2m_jR_j(a_j)\log|z-a_j|^2+O(1)
\qquad \text{as } z\to a_j.
\label{eq:singular-second-derivative-u-theorem13}
\end{align}
Since $R_j(a_j)\neq0$, the asymptotic formula
\eqref{eq:singular-second-derivative-u-theorem13} shows that $\partial_z^2u$ is essentially unbounded in every neighborhood of $a_j$. Hence, $\partial_z^2u\notin L^\infty_{\mathrm{loc}}(V)$ for any open neighborhood
$V\subset U$ of $a_j$.
Therefore,
$u\notin W^{2,\infty}_{\mathrm{loc}}(U)$. This proves (3).
\end{proof}

    \begin{remark}
		The exponent $2$ in the definition $u=P^2\log |f|^2$ is critical. Indeed, if one replaces
		$P^2$ by $P^q$ with an integer $q\ge3$, and defines
		$$
		u_q(z):=P(z)^q\log |f(z)|^2,
		$$
		then $u_q\in W^{2,\infty}_{\mathrm{loc}}(U)$, so this construction no longer gives a
		counterexample to the Calder\'on-Zygmund theorem  for $p=\infty$.
		To see this, it suffices to look near a zero $a_j$ of $f$. Write $\xi=z-a_j$. Then
		$f(z)=\xi^{m_j}h_j(z)$ with $h_j(a_j)\neq0$, and $P(z)=\xi Q_j(z)$ with
		$Q_j(a_j)\neq0$. Hence,
		$$
		u_q(z)=\xi^q Q_j(z)^q\left(m_j\log|\xi|^2+\log|h_j(z)|^2\right).
		$$
		After differentiating twice, the only possible unbounded term is of the form
		$
		C\xi^{q-2}\log|\xi|^2.
		$
		This term is unbounded exactly when $q=2$, while it is bounded for every $q\ge3$.
		Thus, $q=2$ is precisely the borderline exponent: it is large enough to make
		$P^2\log|f|^2$ belong to $C^1$, but not large enough to make its second derivatives
		locally bounded.
	\end{remark}

%\textbf{Acknowledgements}

\end{document}